\documentclass[11pt]{amsart}
\usepackage{amsmath, amscd, amsthm, amsfonts, amssymb, mathrsfs}
\usepackage{epic,eepic,epsf,epsfig}
\usepackage{amsfonts,srcltx,mathrsfs}
\usepackage{amsmath}
\usepackage{amssymb}
\usepackage{amsbsy}
\usepackage{float}
\usepackage{graphicx}
\usepackage{amsfonts}
\usepackage{color}

\usepackage{url}
\newtheorem{lem}{Lemma}[section]
\newtheorem{theorem}[lem]{Theorem}

\newtheorem{case}{Case}
\newtheorem{prop}[lem]{Proposition}

 \def\b{\beta}   
\def\r{\rho} \def\s{\sigma}

\def\olg{\overline G}

 \def\lg{\langle} \def\rg{\rangle}
\def\nd{\mathrel{\bigm|\kern-.7em/}}

\def\f{\noindent}

\def\PSL{\hbox{\rm PSL}}

 \def\PGL{\hbox{\rm PGL}}
\def\Sz{\hbox{\rm Sz}}

\def\Aut{\hbox{\rm Aut}}

\def\Ker{\hbox{\rm Ker}}
\def\Aut{\hbox{\rm Aut}}

\def\PGL{\hbox{\rm PGL}}
\def\PGaL{\hbox{\rm P\mbox{$\Gamma$}L}}

\def\mod{\hbox{\rm mod }}
\def\PGL{\hbox{\rm PGL}}

\def\PGL{\hbox{\rm PGL}}

\def\mz{{\mathbb Z}}
\def\P{\mathcal {P}}

\begin{document}

\title{String C-groups of order $4p^m$}

\author{Dong-Dong Hou}
\address{Dong-Dong Hou, Department of Mathematics, Shanxi Normal University, TaiYuan,
100044, P.R. China}
\email{holderhandsome@bjtu.edu.cn}

\author{Yan-Quan Feng}
\address{Yan-Quan Feng, Department of Mathematics, Beijing Jiaotong University, Beijing,
100044, P.R. China}
\email{yqfeng@bjtu.edu.cn}

\author{Dimitri Leemans}
\address{Dimitri Leemans, D\'epartement de Math\'ematique, Universit\'e Libre de Bruxelles, C.P.216, Boulevard du Triomphe, 1050 Bruxelles Belgium}
\email{leemans.dimitri@ulb.be}

\author{Hai-Peng Qu}
\address{Hai-Peng Qu, Department of Mathematics, Shanxi Normal University, TaiYuan,
100044, P.R. China}
\email{orcawhale@163.com}

\date{}

\begin{abstract}
Let $(G,\{\r_0, \r_1, \r_2\})$ be a string C-group of order $4p^m$ with type $\{k_1, k_2\}$ 
for $m \geq 2$, $k_1, k_2\geq 3$ and $p$ be an odd prime. Let $P$ be a Sylow $p$-subgroup of $G$. 
We prove that $G \cong P \rtimes (\mathbb{Z}_2 \times \mathbb{Z}_2)$, $d(P)=2$, 
and up to duality, $p \mid k_1, 2p \mid k_2$. Moreover, we show that if $P$ is abelian, then $(G,\{\r_0, \r_1, \r_2\})$ is tight and hence known. 
In the case where $P$ is nonabelian, we construct an infinite family of string C-group with type $\{p, 2p\}$ of order $4p^m$ where $m \geq 3$. 
\end{abstract}
\maketitle

\section{Introduction}

A {\em string C-group representation}, or {\em string C-group} for short, is a pair $(G,S)$ where $G$ is a group and $S:=\{\r_0, \ldots, \r_{n-1}\}$ is a finite ordered set of $n$ involutions of $G$ such that
\begin{enumerate}
    \item $\langle S \rangle = G$;
    \item For each $i,j\in \{0, \ldots, n-1\}$ with $|i-j| > 1$, we have $\r_i\r_j = \r_j\r_i$;
    \item For each $I,J \subseteq \{0, \ldots, n-1\}$, we have that 
    $\langle \r_i \; | \; i \in I\rangle \cap \langle \r_j \; | \; j \in J \rangle = \langle \r_k \; | \; k \in I \cap J\rangle$. 
\end{enumerate}
Property 2 above is usually called the {\em string property} while property 3 is usually called the {\em intersection property}.
A pair $(G,S)$ satisfying 1 and 2 above is also called a {\em string group generated by involutions} or {\em sggi} for short.
The cardinality of $G$ is called the {\em order} of $(G,S)$ while the cardinality of $S$ is called the {\em rank} of $(G,S)$.
The {\em Schl\"afli type} of $(G,S)$ is the ordered set $\{o(\r_i\r_{i+1})\; |\; i=0, \ldots, n-2\}$ where $o(g)$ denoted the order of the element $g\in G$.
If $|G| = 2\Pi_{i=0}^{n-2}o(\r_i\r_{i+1})$ then $(G,S)$ is called {\em tight}.

String C-groups permit to construct abstract regular polytopes. The study of abstract regular polytopes has a rich history that has been described comprehensively by McMullen and Schulte~\cite{ARP}. In particular, in the latter reference, the authors prove that abstract regular polytopes and string C-groups are essentially the same by showing on the one hand that, given an abstract regular polytope and a flag of that polytope, one can construct a string C-group, and on the other hand, given a string C-group, one can construct an abstract regular polytope.

It is a natural question to try to describe all pairs $(\P, G)$, where $\P$ is a regular polytope and $G$ is the automorphism group
acting transitively on the flags of $\P$. This is equivalent to describe all pairs $(G,S)$ that are string C-groups for given groups $G$.

Michael Hartley built in~\cite{Halg} an atlas of all string C-groups of order less that 2000 and not equal to 1024. Later on, Gomi {\em et al} \cite{sc1024} determined the non-degenerate\footnote{A string C-group is non-degenerate if its Schläfli type does not contain 2's.} string C-groups of order $1024$. Independently, Marston Conder computed all string C-groups of order at most 2000~\cite{atles1}.
Recently, P. Potočnik {\em et al.}~\cite{GroupOrder6000} computed all regular maps with rotational automorphism group of order at most 6000.

An interesting case consists in the pairs $(G, S)$ with $G$ simple or almost simple.
The atlas~\cite{atles2} contains all regular polytopes whose automorphism group $G$ is an almost simple group such that $H \leq G\leq \Aut(H)$ and $H$ is a simple group of order less than one million. Thanks to the analysing of the data collected in that atlas, striking results have been obtained for the symmetric groups $S_n$ and alternating groups $A_n$. Fernandes~{\em et al.}~\cite{fl,flm,sympolcorr} classified string C-groups of ranks $n-1$, $n-2$, $n-3$ and $n-4$ for $S_n$ respectively, and
Cameron {\em et al.}~\cite{CFLM2017} showed that the highest rank of an abstract regular polytope for $A_n$ is $\lfloor (n-1)/2 \rfloor$ when $n\geq 12$, which was known to be sharp by Fernandes~{\em et al.}~\cite{flm1,flm2}.
More recently, Cameron et al.~\cite{CFL2024} showed that the number of abstract regular polytopes of rank $n-k$ for $S_{n}$ is 
a constant independent of $n$ when $n$ is at least $2k+3$.
For a prime power $q$, results about the highest possible rank of a regular polytope with given automorphism group were obtained for the linear groups $\PSL(2, q)$ in \cite{psl2q}, $\PGL(2,q)$ in \cite{pgl2q}, $\PSL(3, q)$ and $\PGL(3, q)$ in \cite{psl3q}, $\PSL(4, q)$ in \cite{psl4q}, for the Suzuki simple groups $^2B_2(q) = \Sz(q)$ in \cite{suzuki}, for the small Ree groups $^2G_2(q) = R(q)$ in~\cite{LSV2018} and for some symplectic and orthogonal groups in~\cite{brooksbank2021}. Furthermore, Connor {\em et al.}~\cite{soclepsl2q} classified abstract regular polytopes for almost simple groups $G$ with $\PSL(2, q) \leq G \leq \PGaL(2, q)$. For a broader survey on these kind of results, we refer to~\cite{Lee2019}.
Another interesting case is constituted by the string C-groups $(G, S)$ with $G$ soluble.
When it comes to soluble groups, except for abelian groups,
the first families that come to mind are groups of orders $2^n$ or $2^np^m$, where $p$ is an odd prime. Schulte and Weiss mentioned in~\cite{Problem} that the order $2^n$ and $2^np$ prove to be more difficult than others and suggested in~\cite[Problem 30]{Problem} to try to classify string C-groups with these orders.

Conder constructed in~\cite{SmallestPolytopes} an infinite family of string C-groups of Schläfli type $\{4, 4, \cdots, 4\}$, rank $n$ and order $2\cdot 4^{n-1}$. These turn out to be the string C-groups of smallest possible order for a given rank $n\geq 9$ as Conder showed in the same paper. The groups of these string C-groups are in fact 2-groups and the string C-groups are tight.

Hou et al.~\cite{HFL, HFL1} showed that all possible Schl\"afli types can be achieved for string C-groups of order $2^n$ for $n \geq 5$. 
In the case of groups of order $2^np$, families of tight string C-groups of rank three with Schl\"afli type $\{k_1, k_2\}$ (and order $2k_1k_2$), were obtained by Cunningham and Pellicer~\cite{GD2016}.
Hou et al.~\cite{HFL2} constructed an infinite family of string C-groups of rank three, of order $2^np$ with type $\{2^{l_{1}}, 2^{l_{2}}p\}$ where $l_1, l_2 \geq 2$ and $l_1+l_2 \leq n-1$. They also gave a family of string C-groups of rank three and  order $3\cdot2^n$ with type $\{6, 6\}$ for $n \geq 5$.

As a next step, it is natural to wish to construct string C-groups of order $2^np^m$ for $m \geq 2$.
Note that if a string C-group has rank at least three, then $4$ divides its order.
Inspired by this, we assume that $|G|=4p^m$ and we focus on the rank three case.
In this paper, we first prove the following theorem (where $d(G)$ is the size of a smallest generating set of $G$).

\begin{theorem}\label{maintheorem}
Let $m\geq 2$ and let $p$ be an odd prime. Let $(G,\{\r_0, \r_1, \r_2\})$ be a non degenerate string C-group of rank three and order $4p^m$ with Schläfli type $\{k_1, k_2\}$. Let $P$ be a Sylow $p$-subgroup of $G$. Then
\begin{itemize}
\item [\rm {(1)}] $G\cong P \rtimes \lg \r_0,\r_2\rg$;
\item [\rm {(2)}] Up to duality, $p \mid k_1, 2p \mid k_2$;
\item [\rm {(3)}] $d(P)=2$.
\end{itemize}
\end{theorem}

Furthermore, we show that, if $P$ is abelian, then a string C-group representation of $P\rtimes (\mathbb{Z}_2 \times \mathbb{Z}_2)$ is tight, and hence in this case, one can rely on the classification of tight polyhedra of Cunningham and Pellicer~\cite{GD2016}. 

\begin{theorem}\label{maintheorem2}
Let $m\geq 2$ and let $p$ be an odd prime. Let $(G,\{\r_0, \r_1, \r_2\})$ be a non degenerate string C-group of order $4p^m$ with type $\{k_1, k_2\}$. 
Let $P$ be a Sylow $p$-subgroup of $G$. Then $P$ is abelian if and only if $G$ is tight, that is $|G|=2k_1k_2$. Moreover, up to duality and isomorphism, there is a unique string C-group of type $\{2p^{l_1}, p^{l_2}\}$ with $m=l_1+l_2$ and $l_1\leq l_2$. Its group is 
$$G=(\lg x\rg \times \lg y \rg) \rtimes \lg \r_0, \r_2)\rg= \lg \r_0, \r_1 = xy\r_0, \r_2\rg=$$
$$\lg x, y, \r_0, \r_2\; | \; x^{p^{l_1}}, y^{p^{l_2}}, \r_0^2, \r_2^2, [x,y], [\r_0, \r_2], x^{\r_0}=x, x^{\r_2}=x^{-1}, y^{\r_0}=y^{-1}, y^{\r_2}=y^{-1}\rg$$ and $G \cong  (C_{p^{l_1}}\times C_{p^{l_2}})\rtimes (C_2\times C_2).$

\end{theorem}


Finally, in the case where $P$ is nonabelian, we prove the following theorem. The proof of that theorem is constructive, meaning that we give an explicit way to construct such string C-groups.
\begin{theorem}\label{maintheorem3}
For $m\geq 3$, there exists a string C-group $(G,\{\r_0, \r_1, \r_2\})$ of order $4p^m$ and Schläfli type $\{p, 2p\}$, with a nonabelian Sylow $p$-subgroup.
\end{theorem}

Note that the data collected by Conder~\cite{atles1} suggest that a complete classification in that case will be extremely difficult to get.

Let us point out that string C-groups constructed from soluble groups are extremely rare.
For instance, there are 49,910,526,325 groups of order $\leq 2000$~\cite{Groups2000}. These groups are readily accessible in {\sc Magma}~\cite{BCP97} except for those of order 1024. Even though those of order 1024 are not in {\sc Magma}, we know they are all soluble.
One can check that out of all groups of order at most 2000, there are 49,910,525,301  that are soluble groups and 1024 that are non-soluble.
According to the data collected by Conder~\cite{atles1}, the number of string C-group representations of soluble (resp. non-soluble) groups is 4968 (resp. 878).
Hence the ratio (number of string C-group representations/number of groups) is 0.000009\% for soluble groups and 85\% for non-soluble groups.
This shows that it is really not easy to find soluble groups that have string C-group representations.

The paper is organised as follows.
In Section~\ref{background}, we give the necessary background to understand this paper.
In Section~\ref{sec3}, we prove Theorem~\ref{maintheorem}.
In Section~\ref{sec4}, we prove Theorem~\ref{maintheorem2}.
Finally, in Section~\ref{sec5}, we construct rank three string C-groups with group having a non abelian Sylow $p$-subgroup, hence proving Theorem~\ref{maintheorem3}.

\section{Additional background}\label{background}
In this paper we always assume that the string C-group $(G,S)$ is non-degenerate, meaning that there is no 2 in its Schläfli type for, otherwise, the group $G$ is a direct product of two smaller groups. 
The following proposition is
straightforward.
\begin{prop}\label{intersection}
The pair $(G,\{\r_0, \r_1, \r_2\})$ is a string C-group if and only if $G = \langle \r_0,\r_1,\r_2\rangle$, $\r_0\r_2 = \r_2\r_0$ and $\lg \r_0, \r_1\rg \cap \lg \r_1, \r_2\rg = \lg \r_1\rg $.
\end{prop}

We use  standard notation for group theory, as in~\cite{GroupBookss,GroupBook} for example.
In this section we now briefly describe some of the knowledge of group theory we need.

Let $G$ be a group.   We define the {\em commutator\/} $[x, y]$ of elements $x$ and $y$ of $G$
by $[x, y]=x^{-1}y^{-1}xy$. The following results are elementary and so we give them without proof.

\begin{prop}\label{commutator}
Let $G$ be a group. Then, for any $x, y, z \in G$,
\begin{itemize}
\item [{\rm(1)}] $[xy, z]=[x, z]^y[y, z]$, $[x, yz]=[x, z][x, y]^z$;
\item [{\rm(2)}] $[x, y^{-1}]^y=[x, y]^{-1}$, $[x^{-1}, y]^x=[x, y]^{-1}$;
\item [{\rm(3)}] $[x^{-1}, y^{-1}]^{xy}=[x, y]$.
\end{itemize}
\end{prop}

Let $G$ be a group.
The {\em Frattini subgroup}, denoted by $\Phi(G)$, is the intersection of all maximal subgroups of $G$. Obviously, $\Phi(G)$ is a characteristic subgroup of $G$. The following theorem is the well-known Burnside Basis Theorem.

\begin{theorem}{\rm ~\cite[Theorem 1.12]{GroupBookss}}\label{burnside}
Let $G$ be a $p$-group and $|G: \Phi(G)| = p^d$.
\begin{itemize}
\item [(1)] $G/\Phi(G) \cong \mz_p^d$. Moreover, if $N \lhd G$ and $G/N$ is elementary abelian, then $\Phi(G) \leq N$.

\item [(2)] Every minimal generating set of $G$ contains exactly $d$ elements.
\end{itemize}
\end{theorem}

The unique cardinality of all  minimal generating sets of a $p$-group $P$ is called the {\em rank} of $P$,  and denoted by $d(P)$.

\begin{lem}\label{Abeliangroupwithinvolutationauto}
Let $G$ be an abelian group of odd order and $\alpha \in \Aut(G)$. If $o(\alpha) =2$, then
$G=G_1 \times G_2$, where $G_{1}=\{g \in G\; | \; g^{\alpha}=g\}$ and $G_2=\{g \in G\; | \; g^{\alpha}=g^{-1}\}$.
\end{lem}

\begin{proof}
Since $(|G|,2)=1$, we have $s|G|+2t=1$ for some integers $s$ and $t$.
For $g \in G$, we have 
$$g=g^1=g^{2t}=g^t(g^{t})^{\alpha}(g^{-t})^{\alpha}g^t.$$

Since $G$ is abelian and $\alpha^2=1$, we have $(g^t(g^{t})^{\alpha})^{\alpha}=g^t(g^{t})^{\alpha}$, $((g^{-t})^{\alpha}g^t)^{\alpha}=((g^{-t})^{\alpha}g^t)^{-1}$.
It means that $g^t(g^{t})^{\alpha} \in G_1$ and $(g^{-t})^{\alpha}g^t \in G_2$.
Hence $G=G_1G_2$.

On the other hand, it is easy to see that $G_1$ and $G_2$ are both normal subgroups of $G$ and that $G_1 \cap G_2=1$. 
Hence $G=G_1 \times G_2$.
\end{proof}

We will also need Sylow's third theorem so we state it here for clarity.
\begin{theorem}[3rd Sylow's theorem]\label{sylow}
Let $G$ be a group of order $p^nm$ with $(p,m)=1$. Let $n_p(G)$ be the number of Sylow $p$-subgroup of $G$. Then
\begin{enumerate}
    \item $n_p(G)\; | \;m$;
    \item $n_p(G) \equiv 1\; \mod p$;
    \item $n_p(G) = [G:N_G(P)]$ where $P$ is a Sylow $p$-subgroup of $G$.
\end{enumerate}\end{theorem}

Finally, we state a proposition about binomial coefficients that will be used in the proof of one of our theorems.

\begin{prop}\cite[Chapter 1]{MA}\label{enumeration}
Let $n, m, k \in \mathbb{N}$ and let $\tbinom{n}{m}$ denote the binomial coefficient indexed by $n$ and $m$. Then
\begin{itemize}
\item [\rm(1)] $\tbinom{n}{m} =\tbinom{n}{n-m}$;
\item [\rm(2)] $\tbinom{n+2}{m+1} =\tbinom{n+1}{m+1}+\tbinom{n+1}{m}$;
\item [\rm(3)] $\tbinom{n}{m}\tbinom{m}{k}=\tbinom{n}{k}\tbinom{n-k}{m-k}$;
\item [\rm(4)] $\sum\limits_{k=0}^m(-1)^{k}\tbinom{n}{k}=(-1)^m\tbinom{n-1}{m}$;
\end{itemize}
Furthermore, if $m <0$, then $\tbinom{n}{m}=0$.
\end{prop}

\section{Proof of Theorem~\ref{maintheorem}}\label{sec3}

Let $m \geq 2$ and $p$ be an odd prime. Let $(G,\{\r_0, \r_1, \r_2\})$ be a  string C-group of rank three and order $4p^m$ with Schläfli type $\{k_1, k_2\}$.


Let us first prove (1) of Theorem~\ref{maintheorem}.
Let $n_p(G)$ be the number of Sylow $p$-subgroup subgroups of $G$ and let $P$ be a Sylow $p$-subgroup of $G$.
By Theorem~\ref{sylow} (1),
$n_p(G) \mid 4$ and $n_p(G) \equiv  1\ \mod p$. It follows that $n_p(G)=1$ or $4$ in our case as $|G| = 4p^m$.
If $n_p(G)=1$, then $P \unlhd G$. By Theorem~\ref{sylow}(2), if $n_p(G) =4$, then $p=3$ and hence $|G|=4\cdot3^m$.
Set $\Omega=\{P^g\; |\; g \in G\}$. Recall that $|\Omega|=4$.
Consider the action $\phi$ of $G$ on $\Omega$ by conjugation and let $K=\Ker \phi$.
As $|\Omega| = 4$, we have that $G/K \leq S_4$. In fact, since $G$ is transitive on $\Omega$ and $K < P$, we have $|K|=3^{m-1}$.
Then $G/K \cong A_4$.
However, the subgroups generated by all involutions in $A_4$ is $\mz_2 \times \mz_2$.
Since $\r_0^2=\r_1^2=\r_2^2=1$, we have $|G/K|=|\lg \r_0K, \r_1K, \r_2K\rg| \leq 4$.
Note that $|K| = p^{m-1}$. Thus $|G| = 4p^{m-1}$, a contradiction as $|G|=4p^m$. Thus, $n_p(G)=1$, that is, $P \unlhd G$.

We thus have $|\lg \r_0, \r_2\rg|=4$ and $|P|=p^m$. Also $|P \cap \lg \r_0, \r_2\rg|=1$.
It follows that $|P \lg \r_0, \r_2 \rg|=4p^m=|G|$, that is, $G=P\lg \r_0, \r_2\rg$.
In face, since $P \unlhd G$, we have $G=P \rtimes \lg \r_0, \r_2\rg$, finishing the proof of (1).

Let us now prove (2) of Theorem~\ref{maintheorem}.
If, say, $k_1$ is not divisible by $p$, then it is equal to 2 and $G$ is degenerate.
The same holds for $k_2$.
We then only need to show that $k_1$ and $k_2$ cannot be both not divisible by 2.
Note that $G=\lg \r_0, \r_0\r_1, \r_1\r_2\rg$. Consider the quotient group $\olg=G/P=\lg \r_0P, \r_1P, \r_2P \rg$.
If $k_1$ and $k_2$ are both not divisible by 2, then $\r_0\r_1$ and $\r_1\r_2$ both have order a power of $p$. Thus $\r_0\r_1 \in P$ and $\r_1\r_2 \in P$. It follows that $G/P=\lg \r_0P\rg$, and hence $|G/P| = 2$, which is impossible because $|G/P|=4$. Therefore, one at least of $k_1$ or $k_2$ must be divisible by $2p$, finishing the proof of (2).

(3) Assume $\r_1=u\r_0^i\r_2^j$, where $u \in P$ and $i , j \in \{0, 1\}$. Then 
$$1=\r_1^2=u\r_0^i\r_2^ju\r_0^i\r_2^j=uu^{\r_0^i\r_2^j}$$ 
and hence $u^{\r_0^i\r_2^j}=u^{-1}$. 
Since $u \in P$, we have $\r_0^i\r_2^j \ne 1$. Then there exists $\r_0^{i'}\r_2^{j'}$ such that
$\lg \r_0, \r_2\rg=\lg \r_0^i\r_2^j,  \r_0^{i'}\r_2^{j'}\rg$.

Now, let $H= \lg u, u^{ \r_0^{i'}\r_2^{j'}}\rg$. Obviously, $\r_0^{i'}\r_2^{j'} \in N_{G}(H)$. Moreover, 
$$u^{\r_0^i\r_2^j}=u^{-1} \in H, (u^{\r_0^{i'}\r_2^{j'}})^{\r_0^i\r_2^j}=(u^{\r_0^{i}\r_2^{j}})^{\r_0^{i'}\r_2^{j'}}=(u^{-1})^{\r_0^{i'}\r_2^{j'}}=(u^{ \r_0^{i'}\r_2^{j'}})^{-1} \in H.$$ 
Then $\r_0^{i}\r_2^{j} \in N_{G}(H)$. It follows that $H\lg  \r_0^i\r_2^j,  \r_0^{i'}\r_2^{j'}\rg=H\lg \r_0, \r_2\rg \leq G$.

On the other hand, it is easy to see that $\r_1\in H\lg  \r_0, \r_2\rg$. Thus $G=H\lg  \r_0, \r_2\rg$, meaning that $P=H$ and $d(P)=d(H) \leq 2$.

Now, if $d(P)=1$, then $P$ is cyclic. Since $k_1, k_2 \geq 3$, we have $(\r_0\r_1)^2, (\r_1\r_2)^2 \in P$.
It follows that $\lg (\r_0\r_1)^2\rg \cap \lg (\r_1\r_2)^2\rg \ne \langle 1 \rangle$, which is impossible
 because $\lg \r_0, \r_1\rg \cap \lg \r_1, \r_2 \rg=\lg \r_1\rg$.
Thus, $d(P)=2$. This finishes the proof of (3) and thus the proof of Theorem~\ref{maintheorem}. \hfill\qed

\section{Proof of Theorem~\ref{maintheorem2}}\label{sec4}
Let $m \geq 2$ and let $p$ be an odd prime.
Let $(G, \{\r_0, \r_1, \r_2\})$ be a string C-group of rank three and order $4p^m$, with Schläfli type $\{k_1, k_2\}$. Let $P$ be a Sylow $p$-subgroup of $G$. By Theorem~\ref{maintheorem}(1), we know that $P$ is a normal subgroup of $G$.
We now prove that if $P$ is abelian, then $(G, \{\r_0, \r_1, \r_2\})$ must be tight, that is $k_1k_2=2p^m$.

By Lemma~\ref{Abeliangroupwithinvolutationauto}, we have $P=A  \times B \times C \times D$, where

$$A=\{g \in P\; |\;  g^{\r_0}=g, g^{\r_2}=g\}, B=\{g \in P\; |\; g^{\r_0}=g, g^{\r_2}=g^{-1}\},$$
$$C=\{g \in P\; |\; g^{\r_0}=g^{-1}, g^{\r_2}=g\}, D=\{g \in P\; |\; g^{\r_0}=g^{-1}, g^{\r_2}=g^{-1}\}.$$

Moreover at least two of the above subgroups are the identity subgroup, and if two of them are not the identity subgroups, they must be cyclic by Theorem~\ref{maintheorem}(3).

If $P=A$, then $\r_0 \in Z(G)$. Consider the quotient group $\olg=G/\lg \r_0\rg$.
Then $\olg=\lg \overline{\r_1}, \overline{\r_2}\rg$ and has order $2p^m$.
Since $o(\overline{\r_1})=o(\overline{\r_2})=2$, we have $\olg \cong D_{2p^m}$.
It follows that $o(\overline{\r_1}\overline{\r_2})=p^m$ and $d(P)=1$,
which is a contradiction with Theorem~\ref{maintheorem} (3).

If $P=B$, $C$ or $D$, then $\r_0, \r_2$ or $\r_0\r_2 \in Z(G)$. Using arguments similar as in the case $P=A$, we get $d(P)=1$, a contradiction. In fact, if $P \cap A \ne \{1\}$, then $\lg \r_0, \r_2\rg \cap Z(G) \ne \{1\}$. It follows that $d(P)=1$, again contradicting Theorem~\ref{maintheorem}(1).

Since we proved that $P$ cannot be one of $A$, $B$, $C$ or $D$ and since $d(P)=2$, we have $P=K\times L$ with $K,L\in \{A,B,C,D\}$ and $K\neq L$.

If $K=A$, then $\r_1 = xy\r_0^i\r_2^j$ for some $x\in A$ and $y\in L$.
Since $\r_1^2 = 1$, we have that 
$$xy(xy)^{\r_0^i\r_2^j} = x^2yy^{\r_0^i\r_2^j} = 1$$
which is impossible as $x^2\neq 1$.
So $P$ must be one of $B\times C$, $B\times D$ or $C\times D$.

If $P=B \times C$, then
$$G=(\lg x\rg \times \lg y \rg) \rtimes \lg \r_0, \r_2)\rg$$
$$=\lg x, y, \r_0, \r_2 \; | \; x^{p^{l_1}}, y^{p^{l_2}}, \r_0^2, \r_2^2, [x,y], [\r_0, \r_2], x^{\r_0}=x, x^{\r_2}=x^{-1}, y^{\r_0}=y^{-1}, y^{\r_2}=y\rg.$$

Let $\r_1=x^{t_1}y^{t_2}\r_0^{i}\r_2^{j}$.
Then $G=\lg \r_0, \r_1, \r_2\rg=\lg \r_0, x^{t_1}y^{t_2}, \r_2\rg
=\lg x^{t_1}, y^{t_2}, \r_0, \r_2\rg =\lg x^{t_1}, y^{t_2}\rg \rtimes \lg \r_0, \r_2\rg$.
If $p \mid t_1$, then $x^{t_1} \in \Phi(P)$ and $P=\lg y^{t_2}\rg$ which is impossible because $d(P)=2$. Thus, $(t_1,p)=1$. By the same reason, we have $(t_2, p)=1$.
Without loss of generality, one can assume $\r_1=xy\r_0^{i}\r_2^{j}$. Since $\r_1^2=1$, 
it is easy to check that if one of $i$ or $j$ equals 0 we get either that $x^2=1$ or $y^2=1$, a contradiction. Hence we have $i=j=1$ and $\r_1=xy\r_0\r_2$. 

On the other hand, since $(\r_1\r_0)^2=(xy\r_2)^2=y^2$, we have $o(\r_0\r_1)=2o(y)$. 
Since $(\r_1\r_2)^2=(xy\r_0)^2=x^2$, we have $o(\r_1\r_2)=2o(x)$. It follows that
$$|\lg \r_0, \r_1\rg \lg \r_1, \r_2\rg|=
\frac{|\lg \r_0, \r_1\rg| \cdot |\lg \r_1, \r_2\rg |}{|\lg \r_0, \r_1\rg \cap \lg \r_1, \r_2\rg|}
=\frac{2\cdot 2o(y)\cdot 2 \cdot 2o(x)}{2}=8o(x)o(y),$$
which is impossible because $|G|=4o(x)o(y)$.

If $P=B \times D$, then
$$G=(\lg x\rg \times \lg y \rg) \rtimes \lg \r_0, \r_2)\rg$$
$$=\lg x, y, \r_0, \r_2 \; | \; x^{p^{l_1}}, y^{p^{l_2}}, \r_0^2, \r_2^2, [x,y], [\r_0, \r_2], x^{\r_0}=x, x^{\r_2}=x^{-1}, y^{\r_0}=y^{-1}, y^{\r_2}=y^{-1}\rg.$$

Without loss of generality, one can assume $\r_1=xy\r_0^{i}\r_2^{j}$. Since $\r_1^2=1$, 
we have $i=0$ and $j=1$. Then $\r_1=xy\r_2$.

Since $(\r_1\r_0)^2=(xy\r_2\r_0)^2=y^2$, we have $o(\r_0\r_1)=2o(y)$. 
Since $\r_1\r_2=xy$, we have $o(\r_1\r_2)=max\{o(x), o(y)\}$. 
However, if $o(y) >o(x)$, then
 $$|\lg \r_0, \r_1\rg \lg \r_1, \r_2\rg|=
\frac{|\lg \r_0, \r_1\rg| \cdot |\lg \r_1, \r_2\rg |}{|\lg \r_0, \r_1\rg \cap \lg \r_1, \r_2\rg|}
=\frac{2\cdot 2o(y)\cdot 2 \cdot o(y)}{2}=4o(y)^2> 4o(x)o(y)=|G|.$$
Then $o(y)  \leq o(x)$. Thus, $2\cdot o(\r_0\r_1) \cdot o(\r_1\r_2)=2\cdot 2o(y) \cdot o(x)=|G|$.
It follows that $G$ is tight.

If $P=C \times D$, we get a case similar to the previous one, with $\r_1=xy\r_0$ and we arrive to the same conclusion. This is the dual case to the previous one.

Here we have proved that if $P$ is abelian then $(G, \{\r_0, \r_1, \r_2\})$ is tight.

From~\cite[Theorem 1.1]{GD2016}, up to duality, there is a unique string C-group of order $4p^m$ with type $\{p^k, 2p^l\}$ with $m=l+k$ and $l\leq k$.
From the proof above, we have that $l\leq k$. Moreover, since we constructed the unique example, $P$ has to be abelian. Therefore, if $G$ is tight, $P$ has to be abelian.
\hfill\qed

\section{Construction of string C-groups of rank three with a group  of order $4p^m$ and a Sylow $p$-subgroup that is not abelian}\label{sec5}

The main purpose of this section is to prove Theorem~\ref{maintheorem3}.
In order to do so, we will give a constructive proof. To do that,  we need to know some more information about groups of order $p^m$. 
Let $P$ be a $p$-group and let $H_p(P)$ be the
subgroup of $P$ generated by all the elements of order different from $p$. 
The subgroup $H_p(P)$ is called the {\em Hughes
subgroup} of $P$ (see~\cite{hughes,hughes1,hughes2}).
A group $P$ of order $p^m$ and nilpotency class $m-1$ is said to be of maximal class if $m \geq 3$. 
The basic
material about these groups can be found in Blackburn~\cite[pages 83--84]{Blackburn} or Huppert~\cite[Chapter 3]{Huppert}.
If $P$ has an abelian maximal subgroup $A$, and $A=H_{p}(P)$, then, by \cite[Chapter 8, Example 8.3.3]{XQ}), 

$$P=\lg s_1, s_2, \cdots, s_r, s_{r+1}, \cdots, s_{p-1}, \b \rg\ \mbox{with}\ \mbox{the}\ \mbox{following}\ \mbox{relations}$$
$$ s_1^{p^e}, s_2^{p^e}, \cdots, s_r^{p^e}, s_{r+1}^{p^{e-1}}, \cdots s_{p-1}^{p^{e-1}}, \b^p,$$
$$[s_i, s_j],  \ {\rm for} \ 1 \leq i, j \leq p-1,\ \mbox{and}\ s_{k+1}=[s_k, \b],  \ {\rm for} \ 1 \leq k \leq p-2,$$
$$[s_{p-1}, \b]=s_1^{-\tbinom{p}{1}}s_2^{-\tbinom{p}{2}}\cdots s_{p-2}^{-\tbinom{p}{p-2}}s_{p-1}^{-\tbinom{p}{p-1}},$$
where 
$$1 \leq e,1 \leq r \leq p-1,\ m=er+(e-1)(p-r-1)+1 \geq 3.$$
Let $A := \lg s_1, s_2, \cdots, s_{p-1}\rg$. We have that $A$ is abelian and $A \unlhd P$.

Moreover, any $x \in P\backslash  A$ has order $p$.

With the relations given above, it is clear that $P=\lg s_1, \b\rg$. Define
\begin{center}
\begin{tabular}{lll}
$\sigma:$ \qquad & $s_1 \mapsto s_1$ \quad &  and \quad $\beta \mapsto \beta^{-1}$,  \\
$\tau:$ \qquad & $s_1 \mapsto s_1^{-1}$ \quad & and \quad $\beta \mapsto \beta$.  \\
\end{tabular}
\end{center}

In this section we will prove the following theorem.

\begin{theorem}\label{maintheorem4}
Let $G=P\rtimes \lg \sigma, \tau\rg$. 
Let $\r_0:=s_1\tau\sigma$, $\r_1:=\b\tau\sigma$ and $\r_2:=\sigma$ with $s_1, \sigma$ and $\tau$ defined as above.
Then $(G,\{\r_0,\r_1,\r_2\})$ is a string C-group of order $4p^m$ with Schläfli type $\{p, 2p\}$, where $m \geq 3$ and $p$ is an odd prime.
\end{theorem}

Obviously, Theorem~\ref{maintheorem3} is a corollary of the theorem above.

Let $s_{k+1}:=[s_{k}, \b]$ for $k = p-1$ and $p$. Then

$$s_p=[s_{p-1}, \b]=\prod\limits_{i=1}^{p-1}s_i^{-\tbinom{p}{i}},$$
$$s_{p+1}=[s_p,\b]=s_p^{-1}s_p^{\b}=\prod\limits_{i=1}^{p-1}(s_i^{-1})^{-\tbinom{p}{i}}\cdot\prod\limits_{i=1}^{p-1}(s_i^{\beta})^{-\tbinom{p}{i}}$$
$$=\prod\limits_{i=1}^{p-1}(s_i^{-1}s_i^{\beta})^{-\tbinom{p}{i}}=\prod\limits_{i=1}^{p-1}[s_i, \beta]^{-\tbinom{p}{i}}=\prod\limits_{i=1}^{p-1}s_{i+1}^{-\tbinom{p}{i}}.$$

In order to prove Theorem~\ref{maintheorem3}, we first need to prove the following two lemmas.

\begin{lem}\label{lemma1}
For any $1 \leq k \leq p$, we have
\begin{itemize}
\item [{\rm(1)}]$s_k^{\b^{1-k}}=\prod\limits_{i=k}^{p+1}s_i^{\tbinom{p+1-k}{i-k}}$;
\item [{\rm(2)}]$s_{k}^{\sigma}=\left\{
\begin{array}{ll}
s_k^{\b^{1-k}},   & k  \mbox{ is odd,} \\
(s_k^{\b^{1-k}})^{-1},  & k   \mbox{ is even.}
\end{array}
\right.$

\end{itemize}
\end{lem}

\begin{proof}
(1) Since $s_{k+1}=[s_k, \b]$, we have $s_k^{\b}=s_ks_{k+1}$.
Since $o(\b)=p$ and $k-1 \leq p$, we have $p-k+1 \geq 0$.
It follows that
$$s_k^{\b^{1-k}}=s_k^{\b^{p+1-k}}=(s_ks_{k+1})^{\b^{p-k}}=(s_ks_{k+1}^2s_{k+2})^{\b^{p-k-1}}=\cdots $$
$$=s_k^{\tbinom{p+1-k}{0}}s_{k+1}^{\tbinom{p+1-k}{1}}s_{k+2}^{\tbinom{p+1-k}{2}}\cdots s_{p}^{\tbinom{p+1-k}{p-k}}s_{p+1}^{\tbinom{p+1-k}{p+1-k}}.$$

\f (2) The proof uses an induction on $k$. For $k=1$, we have $s_1^{\sigma}=s_1$. For $k=2$, by Proposition~\ref{commutator}, we have $$s_2^{\sigma}
=[s_1, \b]^{\sigma}=[s_1, \b^{-1}]=([s_1, \b]^{-1})^{\b^{-1}}=(s_2^{-1})^{\b^{1-2}}.$$

 Hence this lemma is true for $k=1$ and $2$. We now use induction on $k$ and Proposition~\ref{commutator}, splitting the analysis in two cases, namely the case where $k$ is odd and the case where $k$ is even.
 
If $k$ is odd, then $$s_{k}^{\sigma}=[s_{k-1}^{\sigma}, \b^{\sigma}]=[(s_{k-1}^{\b^{2-k}})^{-1}, \b^{-1}]=[s_{k-1}^{-1}, \b^{-1}]^{\b^{2-k}}=([s_{k-1}, \b]^{\b^{-1}s_{k-1}^{-1}})^{\b^{2-k}}.$$

Recall that $A=\lg s_1, s_2, \cdots, s_{p-1}\rg$ is an abelian maximal subgroup of $G$ and $A \unlhd G$.

It follows that $s_k \in A$ and $s_k^{\b} \in A$ for any $k \geq 1$. Thus,
$$s_{k}^{\sigma}=([s_{k-1}, \b]^{\b^{-1}s_{k-1}^{-1}})^{\b^{2-k}}=s_k^{(s_{k-1}^{-1})^{\b}\b^{1-k}}=s_k^{\b^{1-k}}.$$
If $k$ is even, then
$$s_{k}^{\sigma}=[s_{k-1}^{\sigma}, \b^{\sigma}]=[s_{k-1}^{\b^{2-k}}, \b^{-1}]=[s_{k-1}, \b^{-1}]^{\b^{2-k}}=(([s_{k-1}, \b]^{-1})^{\b^{-1}})^{\b^{2-k}}=(s_k^{\b^{1-k}})^{-1}.$$
This finishes the proof.
\end{proof}

\begin{lem}
$\sigma, \tau \in \Aut(P)$, and $\lg \sigma, \tau \rg \cong \mathbb{Z}_2 \times \mathbb{Z}_2$.
\end{lem}

\begin{proof} Obviously, $\sigma^2=\tau^2=1$ and $\sigma\tau=\tau\sigma$. Therefore $\lg \sigma, \tau\rg \cong \mathbb{Z}_2 \times \mathbb{Z}_2$. 
It remains to prove that $\sigma, \tau \in \Aut(P)$.

We first prove that $\sigma \in \Aut(P)$. Obviously, $\lg s_1^{\s}, \b^{\s}\rg=\lg s_1, \b^{-1}\rg=P$. Thus, we only need to show that the generating set $\{s_1^{\s}, s_2^{\s}, \cdots, s_{p-1}^{\s}, \b^{\s}\}$
of $P$ satisfies
the same relations as $\{s_1, s_2, \cdots, s_{p-1}, \b\}$ in order to conclude that $\sigma \in \Aut(P)$.

Recall that $A=\lg s_1, s_2, \cdots, s_{p-1}\rg$ is abelian and $A \unlhd P$. By Lemma~\ref{lemma1}(2),
$$s_{i}^{\sigma}=\left\{
\begin{array}{ll}
s_i^{\b^{1-i}},   & i \mbox{ is odd,} \\
(s_i^{\b^{1-i}})^{-1},  & i \mbox{ is even.}
\end{array}
\right.$$
Now $o(s_i)=o(s_i^{\s})$, $[s_i^{\s}, s_j^{\s}]=1$ for $1 \leq i, j \leq p-1$.
Moreover, it is easy to see that $s_{k+1}^{\sigma}=[s_k^{\s}, \b^{\s}]$ for $1\leq k \leq p-2$. Thus we only need to show that
$$[s_{p-1}^{\sigma}, \b^{\sigma}]=(s_1^{\sigma})^{-\tbinom{p}{1}}(s_2^{\sigma})^{-\tbinom{p}{2}}(s_3^{\sigma})^{-\tbinom{p}{3}}\cdots (s_{p-1}^{\sigma})^{-\tbinom{p}{p-1}},$$
that is, $s_p^{\sigma}=\prod\limits_{i=1}^{p-1}(s_i^{\sigma})^{-\tbinom{p}{i}}.$

By Lemma~\ref{lemma1}(1), we have
$$s_1^{\sigma}=s_1, s_2^{\sigma}=\prod\limits_{i=2}^{p+1}s_i^{-\tbinom{p-1}{i-2}},
s_3^{\sigma}=\prod\limits_{i=3}^{p+1}s_i^{\tbinom{p-2}{i-3}}, \cdots, s_{p-1}^{\sigma}=
\prod\limits_{i=p-1}^{p+1}s_i^{-\tbinom{2}{i-(p-1)}}.$$

\f Then $\prod\limits_{i=1}^{p-1}(s_i^{\sigma})^{-\tbinom{p}{i}}
=s_1^{u_1}s_2^{u_2}\cdots s_{p+1}^{u_{p+1}}$, where $u_i$ comes from Table~\ref{tableU}.

\begin{table}[]
    \centering
\setlength{\arraycolsep}{5pt}
\begingroup
\renewcommand*{\arraystretch}{1.8}
$$
 \begin{bmatrix}

   \begin{array}{c | c|cccccc}
*            &(s_1^{\sigma})^{-\tbinom{p}{1}}   &(s_2^{\sigma})^{-\tbinom{p}{2}}   &(s_3^{\sigma})^{-\tbinom{p}{3}}   &(s_4^{\sigma})^{-\tbinom{p}{4}}       &\cdots &\cdots &(s_{p-1}^{\sigma})^{-\tbinom{p}{p-1}}\\ \hline
u_1       &-\tbinom{p}{1}      & 0                                                & 0                                                     &0 &\cdots &\cdots         &0                               \\ \hline
u_2       &0&\tbinom{p-1}{0}\tbinom{p}{2}      & 0                                                & 0                                                     &\cdots &\cdots &0                                      \\
u_3        &0&\tbinom{p-1}{1}\tbinom{p}{2}      & -\tbinom{p-2}{0}\tbinom{p}{3}      & 0                                                    &\cdots &\cdots &0                                       \\
u_4        &0&\tbinom{p-1}{2}\tbinom{p}{2}      & -\tbinom{p-2}{1}\tbinom{p}{3}      & \tbinom{p-3}{0}\tbinom{p}{4}        &\cdots &\cdots &0                                        \\
u_5      &0&\tbinom{p-1}{3}\tbinom{p}{2}      & -\tbinom{p-2}{2}\tbinom{p}{3}      & \tbinom{p-3}{1}\tbinom{p}{4}        &\cdots &\cdots &0                                     \\
\vdots    &\vdots&\vdots                                          & \vdots                                         & \vdots                                            &\cdots &\cdots &  \vdots                             \\
\vdots    &\vdots&\vdots                                          & \vdots                                         & \vdots                                            &\cdots &\cdots &   \vdots                                \\
u_{p-1}  &0&\tbinom{p-1}{p-3}\tbinom{p}{2}   & -\tbinom{p-2}{p-4}\tbinom{p}{3}   & \tbinom{p-3}{p-5}\tbinom{p}{4}     &\cdots &\cdots &\tbinom{2}{0}\tbinom{p}{p-1}   \\ \hline
u_p     &0&\tbinom{p-1}{p-2}\tbinom{p}{2}   & -\tbinom{p-2}{p-3}\tbinom{p}{3}   & \tbinom{p-3}{p-4}\tbinom{p}{4}     &\cdots &\cdots &\tbinom{2}{1}\tbinom{p}{p-1} \\ \hline
u_{p+1}&0&\tbinom{p-1}{p-1}\tbinom{p}{2}   & -\tbinom{p-2}{p-2}\tbinom{p}{3}   & \tbinom{p-3}{p-3}\tbinom{p}{4}     &\cdots &\cdots &\tbinom{2}{2}\tbinom{p}{p-1}

  \end{array}

  \end{bmatrix}
$$
\endgroup
    \caption{The values of $u_i$.}
    \label{tableU}
\end{table}

Obviously, $u_1=-\tbinom{p}{1}$. By Proposition~\ref{enumeration}(1), we have
$$u_{p+1}=\sum\limits_{j=2}^{p-1}(-1)^{j}\tbinom{p}{j}=\tbinom{p}{p-1}+
\sum\limits_{j=2}^{\frac{p-3}{2}}\left[(-1)^{j}\tbinom{p}{j}+(-1)^{p-j}\tbinom{p}{p-j}\right]=\tbinom{p}{p-1}=p.$$
\f For $2 \leq i \leq p-1$, we have $u_i=\sum\limits_{j=2}^{i}(-1)^j\tbinom{p+1-j}{i-j}\tbinom{p}{j}$.
For $i=p$, we have $$u_p-\tbinom{1}{0}\tbinom{p}{p}=u_p-1=
\sum\limits_{j=2}^{p}(-1)^j\tbinom{p+1-j}{p-j}\tbinom{p}{j}.$$

By Proposition~\ref{enumeration}$(2)$$(3)$, we have
$$\tbinom{p+1-j}{i-j}\tbinom{p}{j}=\left[\tbinom{p-j}{i-j}+\tbinom{p-j}{i-j-1}\right]\tbinom{p}{j}
=\tbinom{p-j}{i-j}\tbinom{p}{j}+\tbinom{p-j}{i-j-1}\tbinom{p}{j}=
\tbinom{p}{i}\tbinom{i}{j}+\tbinom{p}{i-1}\tbinom{i-1}{j}.$$
Together with Proposition~\ref{enumeration}(4), for $2 \leq i \leq p$, we have
$$\sum\limits_{j=2}^{i}(-1)^j\left[\tbinom{p}{i}\tbinom{i}{j}+\tbinom{p}{i-1}\tbinom{i-1}{j}\right]$$
$$=\tbinom{p}{i}\left[i-1+\sum\limits_{j=0}^{i}(-1)^j\tbinom{i}{j}\right]+
\tbinom{p}{i-1}\left[i-2+\sum\limits_{j=0}^{i}(-1)^j\tbinom{i-1}{j}\right]$$
$$=\tbinom{p}{i}\left[i-1+(-1)^i\tbinom{i-1}{i}\right]+
\tbinom{p}{i-1}\left[i-2+(-1)^i\tbinom{i-2}{i}\right]$$
$$=\tbinom{p}{i}(i-1)+\tbinom{p}{i-1}(i-2)$$
$$=-\tbinom{p}{i}+(p-i+1)\tbinom{p}{i-1}+(i-2)\tbinom{p}{i-1}$$
$$=-\tbinom{p}{i}
+(p-1)\tbinom{p}{i-1},$$
and hence
$$u_p=-\tbinom{p}{p}
+(p-1)\tbinom{p}{p-1}+1=(p-1)\tbinom{p}{p-1},$$
$$u_i=-\tbinom{p}{i}
+(p-1)\tbinom{p}{i-1}, \quad 2 \leq i \leq p-1.$$

By Lemma~\ref{lemma1} and the fact that $s_{p+1}=[s_p,\b]$, we have $s_p^{\sigma}=s_p^{\b^{1-p}}=s_p^{\b}=s_ps_{p+1}.$
The result is now immediate since $$s_1^{u_1}s_2^{u_2}\cdots s_p^{u_p}s_{p+1}^{u_{p+1}}$$
$$=s_1^{-\tbinom{p}{1}}\left[\prod\limits_{i=2}^{p-1}s_i^{-\tbinom{p}{i}
+(p-1)\tbinom{p}{i-1}}\right]s_{p}^{(p-1)\tbinom{p}{p-1}}s_{p+1}^{p}$$
$$=\left[\prod\limits_{i=1}^{p-1}s_i^{-\tbinom{p}{i}}\right]\left[(\prod\limits_{i=2}^{p}s_i^{-\tbinom{p}{i-1}})^{1-p}\right]s_{p+1}^p$$
$$=s_ps_{p+1}^{1-p}s_{p+1}^{p}=s_{p}s_{p+1}=s_p^{\sigma}.$$

Let us now prove that $\tau \in \Aut(P)$. Since $A$ is abelian and $s_{i+1}=[s_i, \b]$ for $1 \leq i \leq p-2$, by Proposition~\ref{commutator}, we have
$$s_{i+1}^{\tau}=[s_i, \b]^{\tau}=[s_i^{-1}, \b]=(([s_i, \b])^{-1})^{s_i^{-1}}=(s_{i+1}^{-1})^{s_i^{-1}}=s_{i+1}^{-1}.$$

We need to show that the generating set $\{s_1^{-1}, s_2^{-1}, \cdots, s_{p-1}^{-1}, \b\}$ of $P$ satisfies
the same relations as $\{s_1, s_2, \cdots, s_{p-1}, \b\}$.
Obviously,  $o(s_i)=o(s_i^{\tau})$, $[s_i^{\tau}, s_j^{\tau}]=1$ for $1 \leq i, j \leq p-1$.
Also, we have $$[s_k, \b]^{\tau}=[s_k^{\tau}, \b^{\tau}]=[s_{k}^{-1},\ \b]=([s_k, \b]^{-1})^{s_k^{-1}}=s_{k+1}^{-1}=s_{k+1}^{\tau}\ \mbox{for}\ 1\leq k \leq p-2.$$

Now we only need to show that
$$[s_{p-1}^{\tau}, \b^{\tau}]=(s_1^{\tau})^{-\tbinom{p}{1}}(s_2^{\tau})^{-\tbinom{p}{2}}(s_3^{\tau})^{-\tbinom{p}{3}}\cdots (s_{p-1}^{\tau})^{-\tbinom{p}{p-1}}.$$
It is easy to see that $$[s_{p-1}^{\tau}, \b^{\tau}]=[s_{p-1}^{-1}, \b]=([s_{p-1}, \b]^{-1})^{s_{p-1}^{-1}}=(s_{p}^{-1})^{s_{p-1}^{-1}}=s_{p}^{-1}.$$ On the other hand,
$$\prod\limits_{i=1}^{p-1}(s_i^{\tau})^{-\tbinom{p}{i}}=\prod\limits_{i=1}^{p-1}(s_i^{-1})^{-\tbinom{p}{i}}
=(\prod\limits_{i=1}^{p-1}s_i^{-\tbinom{p}{i}})^{-1}=s_p^{-1}.$$
It follows that $\tau \in \Aut(P)$.
\end{proof}

We are now ready to prove Theorem~\ref{maintheorem4}.

\begin{proof}[Proof of Theorem~\ref{maintheorem4}]
Let $\r_0:=s_1\tau\sigma$, $\r_1:=\b\tau\sigma$ and $\r_2:=\sigma$. Since $o(\b)=p, o(\tau)=2$ and $[\b, \tau]=1$, we have
$(\r_1\r_2)^p=(\b\tau)^p=\tau$, and hence $\b=(\r_1\r_2)^{p+1}$. It follows that $G=\lg \r_0, \r_1, \r_2\rg$.

Note that
$$\r_0^2=s_1\tau\sigma s_1\tau\sigma=s_1s_1^{\tau\sigma}=s_1s_1^{-1}=1,$$
$$\r_1^2=\b\tau\sigma\b\tau\sigma=\b\b^{\sigma\tau}=\b\b^{-1}=1,$$
$$(\r_0\r_2)^2=(s_1\tau\sigma\sigma)^2=s_1s_1^{\tau}=s_1s_1^{-1}=1.$$
The pair $(G, \{\r_0, \r_1, \r_2\})$ is thus an {\em sggi}. We still need to prove that $(G, \{\r_0, \r_1, \r_2\})$ is a string C-group, or in other words, that it satisfies the intersection property.

Note that for any $x \notin \lg s_1, s_2, \cdots, s_{p-1}\rg$, we have $o(x)=p$.
Since $$\r_0\r_1=s_1\tau\sigma\b\tau\sigma=s_1\b^{-1} \notin \lg s_1, s_2, \cdots, s_{p-1}\rg,$$  $o(\r_0\r_1)=p$.
Since $(\r_1\r_2)^2=(\b\tau\sigma\tau)^2=(\b\sigma)^2=\b^2$, $o(\r_1\r_2)=2p$.
If follows that $|\lg \r_0, \r_1\rg|=2p$ and $|\lg \r_1, \r_2\rg|=4p$.

Finally, if $\lg \r_0, \r_1\rg \cap \lg \r_1, \r_2\rg >\lg \r_1\rg$, then $\lg \r_0, \r_1\rg \leq \lg \r_1, \r_2\rg$ as $\lg \r_1 \rg$ is a maximal subgroup of $\lg \r_0, \r_1\rg$.
It means that $G=\lg \r_1, \r_2\rg$, which is impossible because $|G|=4p^m>4p$.
Thus, $\lg \r_0, \r_1\rg \cap \lg \r_1, \r_2\rg=\lg \r_1\rg$.
By Proposition~\ref{intersection}, $(G, \{\r_0, \r_1, \r_2\})$ is a string C-group
with type $\{p, 2p\}$. 
\end{proof}

\section{Acknowledgements} This work was supported by the National Natural Science Foundation of China (12201371,
12331013,12311530692,12271024,12161141005), the 111 Project of China (B16002), the Fundamental Research Program of Shanxi Province 20210302124078, an Action de Recherche Concertée grant of the Communauté Française Wallonie Bruxelles and a PINT-BILAT-M grant from the Fonds National de la Recherche Scientifique de Belgique (FRS-FNRS).

\end{document}